\documentclass{amsart}
\usepackage{amsmath}        %
\usepackage{amsthm}     %
\usepackage{amscd}      %
\usepackage{amsfonts}       
\usepackage{amssymb}        
\usepackage{wasysym}        
\usepackage{parskip}

\renewcommand{\H}{\mathbb{H}}
\newcommand{\C}{\mathbb{C}}
\newcommand{\R}{\mathbb{R}}
\newcommand{\Z}{\mathbb{Z}}
\newcommand{\N}{\mathbb{N}}

\newcommand{\SL}{\mathrm{SL}}

\newcommand{\mockf}{\mathrm{f}}
\newcommand{\our}{\mathcal{C}}
\newcommand{\wid}{\mathbf{\mathrm{w}}}

\theoremstyle{definition}

\theoremstyle{plain}
\newtheorem{theorem}{Theorem}[section]

\newtheorem{lem}[theorem]{Lemma}
\newtheorem{prop}[theorem]{Proposition}
\theoremstyle{remark}

\makeatletter
\def\imod#1{\allowbreak\mkern10mu({\operator@font mod}\,\,#1)}
\makeatother

\begin{document}
\title[Congruences for mock theta functions]{On certain explicit congruences for mock theta functions}
\author{Matthias Waldherr}
\thanks{The author is supported by Graduiertenkolleg "Global Structures in Geometry und Analysis"}

\address{Mathematical Institute\\University of
Cologne\\ Weyertal 86-90 \\ 50931 Cologne \\Germany}
\email{mwaldher@math.uni-koeln.de}

\begin{abstract}
Recently, Garthwaite-Penniston \cite{garthpennis} have shown that
the coefficients of Ramanujan's mock theta function $\omega$ satisfy
infinitely many congruences of Ramanujan-type. In this work we give
the first explicit examples of congruences for Ramanujan's mock
theta function $\omega$ and another mock theta function $\our$.
\end{abstract}

\maketitle

\section{Introduction and Statement of Results}

The famous "Ramanujan congruences" for the partition function
\begin{align*}
p(5n+4) &\equiv 0 \pmod{5}\\
p(7n+5) &\equiv 0 \pmod{7}\\
p(11n+6) &\equiv 0 \pmod{11}
\end{align*}
found nearly 100 years ago, have since had an enormous impact on the
theory of $q$-series arising in combinatorics. One was trying to
understand why such congruences appear and whether there are
analogous Ramanujan-type congruences for other interesting
$q$-series. The first question is now well understood. The
\textit{partition rank} introduced by Dyson \cite{dyson} explains
the first two congruences from a combinatorial point of view. The
\textit{crank} later defined by Andrews and Garvan \cite{andgarv}
even explains all three congruences. For the second question, many
examples for similar congruences have been found for other
$q$-series, interestingly even for generating series of the rank and
crank itself, see \cite{bringonoCon} and \cite{mahlburg}.

Using computer calculations it is easy to find candidates for such
congruences for a given $q$-series. However, there is no general
method for proving them. In most cases the proofs rely on ingenious
$q$-series manipulations and $q$-series identities. Only in the
special case, when the $q$-series in question is a modular form,
there is an approach, which is capable of proving or disproving any
given congruence. This method is based on the fact - known as
Sturm's theorem \cite{sturm} - that it suffices to check the
congruences for the coefficients of a modular forms up to some
explicitly computable bound in order to conclude that it holds for
all coefficients.

Many interesting examples of $q$-series which arise in combinatorics
are not modular. Famous examples are given by Ramanujan's mock theta
functions
\[ \mockf(\tau):=1+\sum_{n=1}^\infty \frac{q^{n^2}}{(1+q)^2(1+q^2)^2\cdots(1+q^{n})^2}=\sum_{n=0}^\infty a_\mockf(n)q^n\]
and
\[ \omega(\tau):=\sum_{n=0}^\infty \frac{q^{2n^2+2n}}{(1-q)^2(1-q^3)^2\cdots(1-q^{1+2n})^2}=\sum_{n=0}^\infty a_\omega(n)q^n,\]
where $q=e^{2 \pi i \tau}$ with $\tau$ in the upper halfplane $\H$.

Over the years, there has been much effort by authors like Watson
\cite{watson}, Dragonette and Andrews \cite{andrews}
\cite{dragonette} to understand the mock theta functions. A new
chapter in the study of mock theta functions, however, was opened
only recently by Zwegers. In \cite{zwegers} and \cite{zwegersthesis}
he related Ramanujan's mock theta functions to harmonic weak Maass
forms which are certain nonholomorphic generalizations of classical
modular forms (for a precise definition we refer to the next
chapter). This result turned out to be the starting point for
further research in this area. Since questions about asymptotics,
exact formulas and congruences are better understood in the context
of Maass forms and modular forms, it became possible to prove
longstanding conjectures on mock theta functions: In
\cite{bringTAMS} Bringmann proves asymptotic formulas for rank
generating functions. In \cite{bringonoAD} Bringmann-Ono prove the
Andrews-Dragonette conjecture concerning an exact formula for the
coefficients of the mock theta function $\mockf$ and in
\cite{bringonoCon} Bringmann-Ono prove the existence of infinitely
many Ramanujan-type congruences for $\mockf$. Garthwaite
\cite{garth} and Garthwaite-Penniston \cite{garthpennis} have
obtained similar results for $\omega$.

Despite the fact that there are infinitely many congruences to
authors knowledge not a single example has been exhibited yet. Using
Borcherds' products Bruinier-Ono \cite{bruinierono} show congruences
for $\omega$, which are, however, not of Ramanujan-type.

In contrast to that, it is quite easy to find candidates for
congruences simply by computing many coefficients and searching for
congruence patterns. This was done by Jeremy Lovejoy for several
$q$-series including $\omega$ and the following function
\[ \our(\tau):=\sum_{n=0}^\infty (-1)^n \frac{(q;q^2)_n}{(-q;q)_n^2}=2\prod_{n=1}^\infty  \frac{1-q^{2n-1}}{1-q^{2n}}\sum_{n \in \Z}  \frac{q^{\frac{1}{2}n(n+1)}}{1+q^n}=\sum_{n=0}^\infty a_\our(n)q^n, \]
which we will call Cesaro function (since the series representation
has to be interpreted it in the Cesaro sense), and which is also
related to Zwegers' work.

The objective of this paper is twofold. Firstly, we prove
congruences for $\omega$ and $\our$ thereby verifying the
conjectures of Lovejoy and giving the first explicit examples of
congruences for $\omega$. Our results can be summarized in the
following theorem.

\begin{theorem}\label{thm_mainthm}
For all $n \geq 0$ the following congruences hold:
\begin{align*}
a_\omega(40n+27)&\equiv a_\omega(40n+35) \equiv 0 \pmod{5},\\
a_\our(3n+1) &\equiv 0 \pmod{3},\\
a_\our(7n+2) &\equiv a_\our(7n+3) \equiv a_\our(7n+5) \equiv 0
\pmod{7}.
\end{align*}
\end{theorem}

Secondly, in the course of proving the theorem it will turn out that
our approach is capable of proving or disproving a given congruence
of Ramanujan type (i.e the congruences that are supported on
arithmetic progressions) for $\omega$ and $\our$ as long as a
certain condition is satisfied. To explain this condition we note
that both $\our$ and $\omega$ will appear as holomorphic parts of
harmonic weak Maass forms. The condition now requires that the
congruences only involve coefficients of the Fourier expansion of
the Maass form which belong to the holomorphic part. We reduce
proving a congruence to a finite amount of computation in the same
fashion as Sturm's theorem does this for modular forms. In fact, the
application of Sturm's theorem is a crucial step in our argument.

\section*{Acknowledgements}

I thank Kathrin Bringmann for introducing me into this topic, many
interesting discussions and helpful comments on earlier versions of
this paper. I also thank Claudia Alfes for helpful comments on
earlier versions of this paper. Furthermore I thank Jeremy Lovejoy
for sharing with me his computational results. Finally, I would like
to thank Ben Kane and Christian Reitwiessner for helping me to
resolve issues regarding the implementation of the algorithms used
in this work.

\section{Basics facts about harmonic weak Maass forms}
In this paper we use standard terminology of the theory of modular
forms. For basic definitions the reader is referred to chapter 1 of
\cite{ono}. Additionally to modular forms, harmonic Maass forms will
play a crucial role. These functions were originally introduced by
Bruinier-Funke in \cite{bruinierfunke}. An excellent reference for
harmonic Maass forms for the theta multiplier and its applications
is \cite{ono2}. However, in this paper we also work with other
multiplier systems. Therefore we give a definition of modular forms
and harmonic weak Maass forms for arbitrary multiplier systems. A
function $f:\H \to \C$ is called a \textit{weakly holomorphic
modular form } of weight $\frac{k}{2}$ with respect to a congruence
subgroup $\Gamma \leq \SL_2(\Z)$ and a multiplier system $\nu$ if
the following conditions hold:
\begin{enumerate}
\item $f$ satisfies the modular transformation property:
\[ f(\tfrac{a\tau+b}{c\tau+d})=\nu\left(%
\begin{smallmatrix}
  a & b \\
  c & d \\
\end{smallmatrix}%
\right)(c\tau+d)^{\frac{k}{2}}f(\tau).\]
\item $f$ is holomorphic on $\H$.
\item $f$ has at most linear exponential growth at the cusps.
\end{enumerate}
We call a function $f:\H \to \C$ a \textit{harmonic weak Maass
form}, if the second condition is replaced by the weaker condition
that $f$ is annihilated by the weight $\frac{k}{2}$ hyperbolic
Laplacian $\Delta_\frac{k}{2}:=-y^2\left(\frac{\partial^2}{\partial
x^2}+\frac{\partial^2}{\partial
y^2}\right)+\frac{ik}{2}\left(\frac{\partial}{\partial x} +i
\frac{\partial}{\partial y} \right)$. Note that this implies that
$f$ is a real-analytic function. We will write
$M^!_\frac{k}{2}(\Gamma,\nu)$ and $H_\frac{k}{2}(\Gamma,\nu)$ for
the space of weakly holomorphic modular forms and harmonic weak
Maass forms, respectively. By $M_\frac{k}{2}(\Gamma,\nu)$ we denote
the space of holomorphic modular forms, i.e, weakly holomorphic
modular forms which are holomorphic at all the cusps.\\
Let $\chi$ be a Dirichlet character and define the
$\nu_{\theta,\chi}:\Gamma_0(4) \to \{ z \in \C| \ |z|=1\}$ by
\[ \nu_{\theta,\chi}\left(%
\begin{smallmatrix}
  a & b \\
  c & d \\
\end{smallmatrix}%
\right):=\chi(d)\left( \tfrac{c}{d}\right) \epsilon_d^{-1},\] where
$\left( \tfrac{c}{d}\right)$ denotes the Jacobi symbol and
\[\epsilon_d:=\left\{%
\begin{array}{ll}
    1 & \hbox{if } d \equiv 1 \pmod{4}, \\
    i & \hbox{if } d \equiv 3 \pmod{4}.\\
\end{array}%
\right.\] This function is called the $\theta-$multiplier and a
harmonic weak Maass forms with respect to the this multiplier will
just be called a harmonic weak Maass with character $\chi$. The
$\eta$\textit{-multiplier} $\nu_\eta:\SL_2(\Z) \to \{ z \in \C| \
|z|=1\}$ is defined by
\[ \nu_\eta\left(%
\begin{smallmatrix}
  a & b \\
  c & d \\
\end{smallmatrix}%
\right):= \frac{1}{\sqrt{c\tau+d}
}\frac{\eta(\frac{a\tau+b}{c\tau+d})}{\eta(\tau)},\] where $\eta$ is
the classical Dedekind $\eta$-function. Both the $\theta$- and the
$\eta$-multiplier and any integral power of them are multiplier
systems for all half-integral weights.

\section{The Cesaro function $\our$}
In this section we prove the congruences for $\our$ using results
form Zwegers' thesis.
\subsection{Zwegers' results on $\mu$}
We first review some results of Zwegers \cite{zwegersthesis}, who
defines the following function
\[ \mu(u,v;\tau):=\frac{e^{\pi i u}}{\theta(v;\tau)} \sum_{n \in \Z} \frac{(-1)^n e^{\pi i (n^2+n)\tau+2 \pi i nv}}{1-e^{2 \pi i n \tau +2 \pi i u}},\]
for $\tau \in \H$ and $u,v \in \C \setminus (\Z \tau + \Z)$ and
where
\[ \theta(v;\tau):=\sum_{\nu \in \frac{1}{2}+\Z} e^{\pi i \nu^2 \tau+2\pi i \nu(v+\frac{1}{2})}\]
is the classical Jacobi theta function. The function $\mu$ is
holomorphic but does not transform like a modular form. Zwegers
shows that we can complete $\mu$ by adding a non-holomorphic but
real analytic correction term $R$ in the following way
\[ \widetilde{\mu}(u,v;\tau):=\mu(u,v;\tau)+ \frac{i}{2}R(u-v;\tau),\]
so that the resulting function $\widetilde{\mu}$ has nice
transformation properties. We will not recall Zwegers' construction
of the correction term $R$, since we do not use it. We need another
description of $R$ in terms of following unary theta series of
weight $\frac{3}{2}$:
\[ g_{a,b}(\tau):=\sum_{\nu \in a+ \Z} \nu e^{\pi i \nu^2\tau+2 \pi i \nu b},\]
where $a,b \in \R$.
\begin{theorem}[\cite{zwegersthesis}, Theorem 1.16]
For $a \in ]\frac{1}{2},\frac{1}{2}[$ and $b \in \R$ we have
\[ -e^{- \pi i a^2\tau+2 \pi i a(b+\frac{1}{2})}R(a\tau+b;\tau)=\int_{-\overline{\tau}}^{i \infty} \frac{g_{a+\frac{1}{2},b+\frac{1}{2}}(t)}{\sqrt{-i(t+\tau)}}dt.\]
\end{theorem}
Zwegers then proves transformation formulas for $\widetilde{\mu}$.
From his results (Theorem 1.11 of \cite{zwegersthesis}) we easily
obtain the following proposition.
\begin{prop}[]\label{prop_jacobiform}
For any $z \in
\C$, any $\tau \in \H$ and any $\left(%
\begin{smallmatrix}
  a & b \\
  c & d \\
\end{smallmatrix}%
\right) \in \SL_2(\Z)$, we have:
\[ \widetilde{\mu}\left(\tfrac{z}{c\tau+d},\tfrac{z}{c\tau+d};\tfrac{a\tau+b}{c\tau+d}\right)=\nu_\eta^{-3}\left(%
\begin{smallmatrix}
  a & b \\
  c & d \\
\end{smallmatrix}%
\right)(c\tau+d)^{\frac{1}{2}} \widetilde{\mu}(z,z;\tau).\]
Furthermore, for any $z \in \C$ and any $\tau \in \H$ we have
\[ \widetilde{\mu}(z + \tau, z + \tau;\tau)= \widetilde{\mu}(z,z;\tau) \text{ and } \widetilde{\mu}(z + 1,z+1;\tau )= \widetilde{\mu}(z,z;\tau).\]
\end{prop}
\subsection{Modular transformation properties}
The holomorphic part of the function
$\widetilde{\mu}(\frac{1}{2},\frac{1}{2};\tau)$ will turn out to be
related to $\our$. In this section we will study the modular
transformation properties of
$h_1(\tau):=\widetilde{\mu}(\frac{1}{2},\frac{1}{2};\tau)$. For that
purpose we additionally need the functions
$h_2(\tau):=\widetilde{\mu}(\frac{\tau}{2},\frac{\tau}{2};\tau)$ and
$h_3(\tau):=\widetilde{\mu}(\frac{\tau+1}{2},\frac{\tau+1}{2};\tau)$.\\
We use the following abbreviations $T=\left(%
\begin{smallmatrix}
  1 & 1 \\
  0 & 1\\
\end{smallmatrix}%
\right)$ and $S=\left(%
\begin{smallmatrix}
  0 & -1 \\
  1 & 0\\
\end{smallmatrix}%
\right)$.
\begin{lem}\label{lem_hmod}
We have
\[ \mathbf{h}(\tau+1)=\left(%
\begin{array}{ccc}
  \nu_\eta^{-3}(T) & 0 & 0 \\
  0 & 0 & \nu_\eta^{-3}(T) \\
  0 & \nu_\eta^{-3}(T) & 0 \\
\end{array}%
\right)\mathbf{h}(\tau),\]
\[ \frac{1}{\sqrt{\tau}}\mathbf{h}(-\tfrac{1}{\tau})=\left(%
\begin{array}{ccc}
  0 & \nu_\eta^{-3}(S) & 0 \\
  \nu_\eta^{-3}(S) & 0 & 0 \\
  0 & 0 & \nu_\eta^{-3}(S) \\
\end{array}%
\right)\mathbf{h}(\tau),\] where $\textbf{h}:=(h_1,h_2,h_3)^T$.
\end{lem}
\begin{proof}
We only show the second statement. The first one follows
analogously. By means of the first transformation formula in
Proposition \ref{prop_jacobiform} applied to the matrix $S$ and
$z=\frac{\tau}{2}$ we find \[ h_1(-\tfrac{1}{\tau}) =
\widetilde{\mu}(\tfrac{1}{2},\tfrac{1}{2};-\tfrac{1}{\tau})=\widetilde{\mu}(\tfrac{z}{\tau},\tfrac{z}{\tau};-\tfrac{1}{\tau})=\nu_\eta^{-3}(S)\sqrt{\tau}
\widetilde{\mu}(z,z;\tau)=\nu_\eta^{-3}(S)\sqrt{\tau}h_2(\tau).\]
Similarly, using the same transformation formula for $S$ and
$z=-\frac{1}{2}$ we obtain \[ h_2(-\tfrac{1}{\tau}) =
\widetilde{\mu}\left(\tfrac{-\tfrac{1}{\tau}}{2},\tfrac{-\tfrac{1}{\tau}}{2};-\tfrac{1}{\tau}\right)=\widetilde{\mu}(\tfrac{z}{\tau},\tfrac{z}{\tau};-\tfrac{1}{\tau})=
\nu_\eta^{-3}(S)\sqrt{\tau}
\widetilde{\mu}(-\tfrac{1}{2},-\tfrac{1}{2};\tau).\] The second
transformation formula in Proposition \ref{prop_jacobiform} then
shows
\[ h_2(-\tfrac{1}{\tau})=\nu_\eta^{-3}(S)\sqrt{\tau}
\widetilde{\mu}(-\tfrac{1}{2},-\tfrac{1}{2};\tau)=\nu_\eta^{-3}(S)\sqrt{\tau}
\widetilde{\mu}(\tfrac{1}{2},\tfrac{1}{2};\tau)=\nu_\eta^{-3}(S)\sqrt{\tau}h_1(\tau).\]
Finally, applying the first transformation formula in Proposition
\ref{prop_jacobiform} to the matrix $S$ and $z=\frac{\tau-1}{2}$
yields
\begin{align*}
h_3(-\tfrac{1}{\tau})&=\widetilde{\mu}\left(\tfrac{-\tfrac{1}{\tau}+1}{2},\tfrac{-\tfrac{1}{\tau}+1}{2};-\tfrac{1}{\tau}\right)=\widetilde{\mu}(\tfrac{z}{\tau},\tfrac{z}{\tau};-\tfrac{1}{\tau})
=\nu_\eta^{-3}(S)\sqrt{\tau} \widetilde{\mu}(z,z;\tau)\\
&=\nu_\eta^{-3}(S)\sqrt{\tau}
\widetilde{\mu}(\tfrac{\tau-1}{2},\tfrac{\tau-1}{2};\tau)=\nu_\eta^{-3}(S)\sqrt{\tau}
\widetilde{\mu}(\tfrac{\tau+1}{2};\tfrac{\tau+1}{2}\tau)=\nu_\eta^{-3}(S)\sqrt{\tau}
h_3(\tau).
\end{align*}
\end{proof}

The lemma shows that $\mathbf{h}$ transforms as a vector valued
modular form. We will use this fact to derive the transformation
properties for the first component $h_1(\tau)$.

\begin{prop}\label{prop_modforh}
The function $h_1(\tau)$ is a harmonic weak Maass form of weight
$\tfrac{1}{2}$ with respect to the group $\Gamma_0(2)$ with
multiplier system $\nu_\eta^{-3}$.
\end{prop}
\begin{proof}
Using Lemma \ref{lem_hmod} one readily sees that under any of the
following transformations $-I=\left(%
\begin{smallmatrix}
  -1 & 0 \\
  0 & -1 \\
\end{smallmatrix}%
\right)$, $T=\left(%
\begin{smallmatrix}
  1 & 1 \\
  0 & 1\\
\end{smallmatrix}%
\right),$ and $ST^2S=\left(%
\begin{smallmatrix}
  -1 & 0 \\
  2 & -1 \\
\end{smallmatrix}%
\right)$ the function $h_1$ is mapped to a constant multiple of
itself, where the factor is given by the multiplier system
$\nu_\eta^{-3}$. It is easy to see that the group generated by
$-I,T,ST^2S$ is $\Gamma_0(2)$.\\
The growth condition for $h_1(\tau)$ can be deduced from its Fourier
expansion which we give in the next section. Furthermore it follows
from Zwegers' results (see Proposition 4.2 in \cite{zwegersthesis})
that $h_1(\tau)$ is annihilated $\Delta_\frac{1}{2}$.
\end{proof}

It is important to understand how $h_1(\tau)$ transforms under all
the other elements of $\SL_2(\Z)$.

\begin{prop}\label{prop_cesaromap}
Let $M=\left(%
\begin{smallmatrix}
  a & b \\
  c & d \\
\end{smallmatrix}%
\right) \in \SL_2(\Z)$. Then
$\frac{1}{\sqrt{c\tau+d}}h_1\left(\frac{a\tau+b}{c\tau+d}\right)$ is
a constant multiple of
\begin{enumerate}
\item $h_1(\tau)$, if $c$ is even and $d$ is odd,
\item $h_2(\tau)$, if $c$ is odd and $d$ is even,
\item $h_3(\tau)$, if $c$ and $d$ are odd.
\end{enumerate}
\end{prop}
\begin{proof}
The first statement follows from Proposition \ref{prop_modforh}. If
$c$ is odd and $d$ is even, then
\[ M=\left(%
\begin{smallmatrix}
  a & b \\
  c & d \\
\end{smallmatrix}%
\right)=-\left(%
\begin{smallmatrix}
  a & b \\
  c & d \\
\end{smallmatrix}%
\right)SS=\left(%
\begin{smallmatrix}
  -b & a \\
  -d & c \\
\end{smallmatrix}%
\right)S=(-MS)S\] and we see that $-MS \in \Gamma$. Now the result
follows from Lemma \ref{lem_hmod}. Analogously, if both $c$ and $d$
are odd, then we find $M=(-MT^{-1}S)ST$ and $-MT^{-1}S \in \Gamma$
and the result follows from Lemma \ref{lem_hmod}.
\end{proof}
\subsection{Decomposition into holomorphic and non-holomorphic part and Fourier expansions}
In this section we will see how $\our$ is related to
$\mu(\tfrac{1}{2},\tfrac{1}{2};\tau)$. It follows from Zwegers'
results that
\[ \widetilde{\mu}\left(\tfrac{1}{2},\tfrac{1}{2};\tau\right)=\mu\left(\tfrac{1}{2},\tfrac{1}{2};\tau\right)-\frac{i}{2}\int_{-\overline{\tau}}^{i \infty} \frac{g_{\frac{1}{2},\frac{1}{2}}(t)}{\sqrt{-i(t+\tau)}}dt.\]

From this representation we can easily deduce the Fourier expansion
of the non-holomorphic part of
$\widetilde{\mu}(\tfrac{1}{2},\tfrac{1}{2};\tau)$.

\begin{prop}\label{prop_nonholmorphicofh}
The function $\tilde{\mu}(\tfrac{1}{2},\tfrac{1}{2};\tau)$ has the
following decomposition into a holomorphic and a non-holomorphic
part:
\[ \widetilde{\mu}\left(\tfrac{1}{2},\tfrac{1}{2};\tau\right)=\mu\left(\tfrac{1}{2},\tfrac{1}{2};\tau\right)+  \frac{i\sqrt{2}}{4 \sqrt{\pi}} \sum_{n \in \Z} (-1)^n q^{-\frac{(2n+1)^2}{4}} \Gamma\left(\tfrac{1}{2},\pi(2n+1)^2y\right).\]
\end{prop}
We next turn to the Fourier expansion of the holomorphic part.
Indeed, we require the Fourier expansions of the holomorphic parts
of $h_1,h_2,$ and $h_3$. It is clear that these are
$\mu(\frac{1}{2},\frac{1}{2};\tau),\mu(\frac{\tau}{2},\frac{\tau}{2};\tau),$
and $\mu(\frac{\tau+1}{2},\frac{\tau+1}{2};\tau)$, respectively.

\begin{prop}\label{prop_relationtocesaro}
\begin{enumerate}
\item We have
\[\mu(\tfrac{1}{2},\tfrac{1}{2};\tau)=-\frac{i}{\sum_{n \in
\Z}q^{\frac{1}{2}(n+\frac{1}{2})^2}}\sum_{n \in \Z}
\frac{q^{\frac{1}{2}n(n+1)}}{1+q^n}.\]
 Consequently the Fourier
expansion of $\mu(\tfrac{1}{2},\tfrac{1}{2};\tau)$ starts with
\[ -\frac{i}{4} q^{-\frac{1}{8}}- \frac{3i}{4}q^\frac{7}{8}+\cdots.\]
Furthermore we have
\[ \mu(\tfrac{1}{2},\tfrac{1}{2};\tau)=\frac{1}{4i} q^{-\frac{1}{8}}\our(\tau). \]
\item We have
\[\mu(\tfrac{\tau}{2},\tfrac{\tau}{2};\tau)=-\frac{iq^\frac{1}{4}}{\sum_{n
\in \Z} (-1)^nq^{\frac{1}{2}(n^2+2n+1)}} \sum_{n \in \Z}
\frac{(-1)^n q^{\frac{1}{2}n(n+2)}}{1-q^{n+\frac{1}{2}}}.\]
Consequently the Fourier expansion of the holomorphic part of $h_2$
starts as
\[\mu(\tfrac{\tau}{2},\tfrac{\tau}{2};\tau)=2iq^\frac{1}{4}+6iq^\frac{3}{4}+\cdots.\]
\item We have
\[
\mu(\tfrac{\tau+1}{2},\tfrac{\tau+1}{2};\tau)=-\frac{iq^\frac{1}{4}}{\sum_{n
\in \Z} q^{\frac{1}{2}\pi i (n^2+2n +1) }} \sum_{n \in \Z} \frac{
q^{\frac{1}{2}n(n+2)}}{1+q^{n+\frac{1}{2}}}.\] Consequently the
Fourier expansion of the holomorphic part of $h_3$ starts with
\[ \mu(\tfrac{\tau+1}{2},\tfrac{\tau+1}{2};\tau)=-2iq^\frac{1}{4}+6iq^\frac{3}{4}+\cdots.\]
\end{enumerate}
\end{prop}
\begin{proof}
We find that
\[\mu(\tfrac{1}{2},\tfrac{1}{2};\tau)=\frac{e^{\pi i
\frac{1}{2}}}{\theta(\frac{1}{2};\tau)} \sum_{n \in \Z} \frac{(-1)^n
e^{\pi i (n^2+n)\tau+\pi i n}}{1-e^{2 \pi i n \tau +\pi
i}}=\frac{i}{\theta(\frac{1}{2};\tau)} \sum_{n \in \Z} (-1)^{2n}
\frac{q^{\frac{1}{2}n(n+1)}}{1+q^n}.\]

For the $\theta$-function we obtain the following series expansion
\[ \theta(\tfrac{1}{2};\tau)=\sum_{n \in \Z} e^{\pi i (n+\frac{1}{2})^2 \tau+2\pi i (n+\frac{1}{2})(\frac{1}{2}+\frac{1}{2})}=
-\sum_{n \in \Z} e^{\pi i (n+\frac{1}{2})^2 \tau}=-\sum_{n \in
\Z}q^{\frac{1}{2}(n+\frac{1}{2})^2}.\]

Using the well known product expansion for the $\theta$-function
\[ \theta(v;\tau)=-ie^\frac{ \pi i \tau}{4} e^{-\pi i v}\prod_{n=1}^\infty \left(1-e^{2 \pi i n\tau}\right)\left(1-e^{2 \pi i v}e^{2 \pi i (n-1)\tau}\right)\left(1-e^{-2 \pi i v}e^{2 \pi i n \tau}\right),\]
 we can also write
\[ \theta(\tfrac{1}{2};\tau)=-2q^\frac{1}{8}\prod_{n=1}^\infty
\frac{1-q^{2n}}{1-q^{2n-1}}.
\]
This proves the relation to $\our$. For the holomorphic parts of
$h_2$ and $h_3$ we can derive the Fourier expansions in a similar
manner.
\end{proof}
\subsection{Sieving out residue classes} Since the holomorphic part
of $h_1(\tau)$ is (up to a constant) $q^{-\frac{1}{8}}\our(\tau)$ we
find that the coefficients $a_\our(3n+1)$ correspond to the
coefficients of $h_1(\tau)$ with exponent $24n+7$ in the Fourier
expansion given in terms of $q^\frac{1}{8}$. Similarly, we see that
$a_\our(7n+2),a_\our(7n+3)$, and $a_\our(7n+5)$ correspond to the
Fourier coefficients with exponent $56n+15,56n+23$ and $56n+39$.\\
For a harmonic weak Maass form with Fourier expansion
\[ f(\tau)=\sum_{n \in \Z} a_n(y) q^\frac{n}{\wid} \]
with $\wid \in \N$ and for $r,m \in \Z$ we define the sieve operator
by
\[ U_{r,m} f (\tau):= \sum_{n \equiv r \imod{m}} a_n(y) q^\frac{n}{\wid}.\]

For harmonic weak Maass forms with respect to the
$\theta$-multiplier one can show in a standard manner that
$U_{r,m}f$ again is a harmonic weak Maass form. This result may be
generalized to other multiplier systems. We only give the statement
for the multiplier system $\nu_\eta^{-3}$, which will be needed in
this paper.

\begin{prop}\label{prop_sieveroots}
Let $f$ be a harmonic weak Maass form of weight $\frac{1}{2}$ for
$\Gamma_0(2)$ with respect to $\nu_\eta^{-3}$. Suppose that $f$ has
a Fourier expansion in terms of $q^\frac{1}{8}$. Let $r,m \in \N$.
Then $U_{r,m} f$ is a harmonic weak Maass form of weight
$\frac{1}{2}$ with respect to the congruence subgroup
\[ \Gamma:=\left\{\left. \left(%
\begin{smallmatrix}
  a & b \\
  c & d \\
\end{smallmatrix}%
\right) \in \SL_2(\Z)\right|a,d \text{ coprime to } m, a \equiv d
\pmod{m} \text{ and }2m^2|c\right\}.\] We have $\Gamma_1(2m^2) \leq
\Gamma \leq \Gamma_0(2m^2)$ and
\[ \left[\SL_2(\Z):\Gamma\right]=\frac{2m^4\varphi(m)}{\varphi(2m^2)}\prod_{p|2m^2}\left(1-\tfrac{1}{p^2}\right),\]
where $\varphi$ is Euler's $\varphi$-function.
\end{prop}

If we apply this result to the Fourier coefficients of $h_1$ we are
interested in we obtain the following result.

\begin{prop}\label{prop_cesarosieve}
\begin{enumerate}
\item The function \[ \our_{7(24)}(\tau):=\sum_{n \equiv 7 \imod{24}}
a_\our(n)q^\frac{n}{8}\] is a weakly holomorphic modular form of
weight $\frac{1}{2}$ with respect to $\nu_\eta^{-3}$ for some
subgroup $\Gamma$ of $\SL_2(\Z)$ which contains $\Gamma_1(1152)$ and
has index $9216$ in $\SL_2(\Z)$.
\item The function \[ \our_{15,23,39(56)}(\tau):=\sum_{n \equiv 15,23,39 \imod{56}}
a_\our(n)q^\frac{n}{8}\] is a weakly holomorphic modular form of
weight $\frac{1}{2}$ with respect to $\nu_\eta^{-3}$ for some
subgroup $\Gamma'$ of $\SL_2(\Z)$ which contains $\Gamma_1(6272)$
and has index $129024$ in $\SL_2(\Z)$.
\end{enumerate}
\end{prop}
\begin{proof}
We only show the first assertion, the second on is proven similarly.
The first function is (up to a constant) the holomorphic part of the
function obtained from $h_1(\tau)$ by sieving out the coefficients
$24n+7$ with $n \in \Z$. By Proposition \ref{prop_sieveroots}  and
Proposition \ref{prop_modforh} this function is a harmonic weak
Maass form for the congruence subgroup with the properties stated
above. It remains to show that the function is holomorphic, i.e.,
its non-holomorphic part vanishes. This is due to the fact that by
Proposition \ref{prop_nonholmorphicofh} the non-holomorphic part is
supported at Fourier coefficients of the form
$q^\frac{-2(2n+1)^2}{8}$ with $n \in \Z$ and the fact that $24n+7$
can never be twice a negative square of an odd number. To see this,
suppose that $24n+7\equiv -2(2X+1)^2 \imod{8}$ with $X \in \Z$. This
implies that $7\equiv-2 \imod{8}$, a contradiction.
\end{proof}

We now would like to apply Sturm's theorem to the modular forms in
Proposition \ref{prop_cesarosieve}. However, Sturm's theorem is not
directly applicable since it is only valid for modular forms which
are also holomorphic at the cusps. In the next two sections we will
study the behavior of the functions in Proposition
\ref{prop_cesarosieve} at the cusps and we will prove that Sturm's
theorem can be applied if we first multiply with suitable cusp
forms.
\subsection{Behaviour at the cusps} The functions
$\our_{7(24)}$ and $\our_{15,23,39(56)}$ arise from $h_1(\tau)$ by
applying the sieve operator. For an arbitrary harmonic weak Maass
form $f$ with Fourier expansion
\[ f(\tau)=\sum_{n \in \Z} a_n(y) q^\frac{n}{\wid} \]
the sieve operator $U_{r,m}$ can be written as
\[ U_{r,m} f(\tau)=\frac{1}{m}\sum_{s \imod{m}} \zeta_m^{rs}f\left( \tau-\frac{\wid s}{m}\right),\]
where $\zeta_m:=e^\frac{2 \pi i}{m}$. Hence, in order to determine
the behavior of the two functions above at the cusps, we investigate
how the holomorphic part of $h_1$ behaves under the translation
$\tau \mapsto \tau-\frac{\wid s}{m}$. For this we require the
following lemma which can be proved by a straightforward
calculation.

\begin{lem}\label{lem_cesaromatrixlemma}
Let $\left(\begin{smallmatrix}  a & b \\  c & d
\\ \end{smallmatrix}\right) \in  \SL_2(\Z)$ and suppose $0 \leq  s < m$. Set  $l:=\gcd(cm,\wid sc+am)$. Then define
$\tilde{a}:=\frac{\wid sc+am}{l}$ and $\tilde{c}:=\frac{cm}{l}$. Let
$\tilde{d}$ be such that $\tilde{a}\tilde{d}-1 \equiv 0
\imod{\tilde{c}}$ and
$\tilde{b}:=\frac{\tilde{a}\tilde{d}-1}{\tilde{c}}$. Finally define
$t:=\frac{dl-\tilde{d}m}{c}$. Then $\left(\begin{smallmatrix}
\tilde{a} & \tilde{b} \\  \tilde{c} & \tilde{d}
\\ \end{smallmatrix}\right) \in  \SL_2(\Z)$ and we have for all
$\tau \in \H$
\[  \left(%
\begin{array}{cc}
  1 & \frac{s}{m} \\
  0 & 1 \\
\end{array}%
\right)\left(%
\begin{array}{cc}
  a & b \\
  c & d \\
\end{array}%
\right)\tau\ =
\left(
                         \begin{array}{cc}
                            \tilde{a} & \tilde{b} \\
                            \tilde{c} & \tilde{d} \\
                          \end{array}
                        \right)\left(
     \begin{array}{cc}
       l & t \\
       0 & \frac{m^2}{l} \\
     \end{array}
   \right)\tau. \]
\end{lem}

We can now describe explicitly the pole orders of the holomorphic
parts of the shifted versions of $h_1(\tau)$.
\begin{lem}\label{lem_cesarocuspbehaviour}
Let $\frac{a}{c}$ be a cusp and suppose that
$\left(\begin{smallmatrix} a & b \\  c & d
\\ \end{smallmatrix}\right) \in  \SL_2(\Z)$ is a matrix which maps
$\infty$ to this cusp. Furthermore suppose $m$ is some integer and
$0 \leq s < m$.

\begin{enumerate}
\item The Fourier expansion of the holomorphic part of $h_1(\tau+\frac{\wid s}{m})$
at the cusp $\frac{a}{c}$ is the Fourier expansion at $\infty$ of
$h_1(\tau)$ under the transformation
\[ \left(
                         \begin{array}{cc}
                            \tilde{a} & \tilde{b} \\
                            \tilde{c} & \tilde{d} \\
                          \end{array}
                        \right)\left(
     \begin{array}{cc}
       l & t \\
       0 & \frac{m^2}{l} \\
     \end{array}
   \right)\]
with notations as in Lemma \ref{lem_cesaromatrixlemma}.
\item The first Fourier coefficient of the holomorphic part of $h_1(\tau)$ at the cusp
$\frac{a}{c}$ up to a non-zero constant is given by
\begin{enumerate}
\item $q^{-\frac{l^2}{8m^2}}$ if $\tilde{c}$ is even and $\tilde{d}$ is
odd,
\item $q^{\frac{l^2}{4m^2}}$ if $\tilde{c}$ is odd and $\tilde{d}$ is
even,
\item $q^{\frac{l^2}{4m^2}}$ if $\tilde{c}$ is odd and $\tilde{d}$ is odd.
\end{enumerate}
\end{enumerate}
\end{lem}
\begin{proof}
The statement (1) is clear. The three statements in (2) are proved
by applying Proposition \ref{prop_cesaromap} as follows: If
$\tilde{c}$ is even and $\tilde{d}$ is odd, then by Proposition
\ref{prop_cesaromap} we know that under $\left(\begin{smallmatrix}
\tilde{a} & \tilde{b}
\\  \tilde{c} & \tilde{d}
\\ \end{smallmatrix}\right)$ the function $h_1(\tau)$ is mapped to a multiple of $h_1(\tau)$. The
holomorphic part of $h_1(\tau)$ is up to a constant
$q^{-\frac{1}{8}}\our(\tau)$, which has a Fourier expansion starting
with $q^{-\frac{1}{8}}$. Hence, applying the matrix
$\left(\begin{smallmatrix}  l & t \\
       0 & \frac{m^2}{l} \\
\\ \end{smallmatrix}\right)$ shows that the Fourier expansion starts
with $q^{-\frac{l^2}{8m^2}}$ in this case.\\
In the other two cases
$h_1(\tau)$ is mapped to a multiple of $h_2(\tau)$ or $h_3(\tau)$
under $\left(\begin{smallmatrix} \tilde{a} & \tilde{b}
\\  \tilde{c} & \tilde{d}
\\ \end{smallmatrix}\right)$ by Proposition \ref{prop_cesaromap}. The Fourier
expansion of the holomorphic part of both these functions start with
$q^{\frac{1}{4}}$ and applying the matrix $\left(\begin{smallmatrix}  l & t \\
       0 & \frac{m^2}{l} \\
\\ \end{smallmatrix}\right)$ gives the result.
\end{proof}
\subsection{Application of Sturm's theorem}
In this section we will reduce the task of proving the congruences
for $\our$ to a finite computation by applying Sturm's theorem.
Before we do this we have to get rid of the poles by multiplying
with a suitable cusp form.

\begin{prop}\label{prop_cesarocuspform}
We have
\begin{enumerate} \item  The function
$\eta^{12}(24\tau)\Delta(\tau)$ is a cusp form of weight $18$ for
$\Gamma_0(1151)$. Furthermore the function
\[ \widetilde{\our}_{7(24)}(\tau):=\eta^{12}(24\tau)\Delta(\tau) \sum_{n \equiv 7 \imod{24}}
a_\our(n)q^\frac{n}{8}\] is holomorphic at every cusp of
$\Gamma_1(1152)$.
\item  The function
$\eta^{48}(56\tau)\Delta^2(\tau)$ is a cusp form of weight $48$ for
$\Gamma_0(6272)$. Furthermore the function
\[ \widetilde{\our}_{15,23,39(56)}(\tau):=\eta^{48}(56\tau)\Delta^2(\tau) \sum_{n \equiv 15,23,39 \imod{24}}
a_\our(n)q^\frac{n}{8}\] is holomorphic at every cusp of
$\Gamma_1(6272)$.
\end{enumerate}
\end{prop}
\begin{proof}
The statements for the $\eta$-product follow from Theorem 1.65 of
\cite{ono}. In the first case we take $r_{24}=12$, $r_1=12$, and
$r_d=0$ for all other divisors $d$ of $N=1152$ as in the statement
of the theorem. In the second case we choose $r_{56}=48$, $r_1=24$,
and again $r_d=0$ for all other divisors of $N=6272$.\\
The assertion about the holomorphicity at the cusps is checked by a
computer as follows: For any cusp of $\Gamma_1(1152)$ and
$\Gamma_1(6272)$ respectively we find a representative in the form
$\frac{a}{c}$. Then we use Lemma \ref{lem_cesarocuspbehaviour} to
find upper bounds for the pole orders of  $\sum_{n \equiv 7
\imod{24}} a_\our(n)q^\frac{n}{8}$ and $\sum_{n \equiv 15,23,39
\imod{24}} a_\our(n)q^\frac{n}{8}$ at the cusp. Using Theorem 1.65
of \cite{ono} we may compare the pole orders to the order of
vanishing of the $\eta$-product.
\end{proof}

The following result now turns out to be an easy consequence of
Proposition \ref{prop_cesarocuspform} and Sturm's theorem.

\begin{prop}\label{prop_cesarosturm}
\begin{enumerate}
\item The congruence $a_\our(3n+1) \equiv 0 \pmod{3}$
holds for all $n \in \N$ if it holds for all $n \in \N$ with  $24n+7
\leq 7104$.
\item The congruence $a_\our(7n+2) \equiv
a_\our(7n+3) \equiv a_\our(7n+5) \equiv 0 \pmod{7}$ holds for all $n
\in \N$ if it holds for all $n \in \N$ with $56n+15,56n+23,56n+39
\leq 260734$.
\end{enumerate}
\end{prop}
\begin{proof}
We consider the functions $\widetilde{\our}_{7(24)}(\tau)$ and
$\widetilde{\our}_{15,23,39(56)}(\tau)$. First note definition of
these functions as products we find that they have integral Fourier
coefficients because the individual factors have. Using Proposition
\ref{prop_cesarocuspform} and Proposition \ref{prop_cesarosieve} the
first product is a modular form of weight $18+\frac{1}{2}$ for some
subgroup $\Gamma$ of $\SL_2(\Z)$ satisfying $\Gamma_1(1152) \leq
\Gamma$ and $[\SL_2(\Z):\Gamma]=9216$. Furthermore Proposition
\ref{prop_cesarocuspform} implies that this form is holomorphic at
the cusps of $\Gamma_1(1152)$ and hence also at all cusps of
$\Gamma$. Completely analogously we find, using Proposition
\ref{prop_cesarocuspform} and Proposition \ref{prop_cesarosieve},
that the second product is a modular form of weight $48+\frac{1}{2}$
for a group $\Gamma'$ satisfying $[\SL_2(\Z):\Gamma']=129024$, and
is holomorphic at the cusps. To both forms we apply Sturm's theorem
(see \cite{sturm} Theorem 1), which states that all coefficients of
a modular form of weight $\tfrac{k}{2}$ on $\Gamma$ are divisible by
a prime $p$ iff this holds for the all coefficients up to the
explicit bound
\[ \frac{k}{24}[\SL_2(\Z):\Gamma].\] In our case, this implies that
all coefficients of the first form are divisible by $3$, if this is
true for the first $7104$ ones. For the second form we find that the
first $260736$ coefficients have to be checked. The number
$a_\our(3n+1)$ appears as the coefficient $q^\frac{24n+7}{8}$ in the
expansion of \[ \our_{7(24)}(\tau)=\sum_{n \equiv 7 \imod{24}}
a_\our(n)q^\frac{n}{8}.
\] Suppose the assertion in the statement of the theorem is
satisfied, i.e,  the first 7104 coefficients of $\our_{7(24)}$ are
divisible by 3. Then, also the first 7104 coefficients of
\[ \eta^{12}(24\tau)\Delta(\tau) \sum_{n \equiv 7 \imod{24}}
a_\our(n)q^\frac{n}{8}\] are divisible by 3. But then the
consequence of Sturm's theorem is that all coefficients of the
product are divisible by 3. Since the leading coefficient of
$\eta^{12}(24\tau)\Delta(\tau)$ is $1$, we can now argue with
induction and conclude that all coefficients of $\our_{7(24)}$ are
divisible by 3. This establishes the first claim. The second claim
follows analogously.
\end{proof}

\section{The mock theta function $\omega$}
In this section we complete the proof of Theorem \ref{thm_mainthm}
by proving the congruences for $\omega$. We will carry out a similar
program as for $\our$. First we recall how $\omega$ relates to
Zwegers' results and that it can be seen as the holomorphic part of
a harmonic weak Maass form. Then again we will sieve out for this
Maass form those coefficients which are related to the presumed
congruences. It will turn out that the function obtained in  this
way is a weakly holomorphic modular form. Then we will study its
behavior at the cusps and apply Sturm's theorem. Since this program
so closely parallels our treatment of $\our$, we will omit most of
the proofs.

\subsection{Work of Zwegers and Garthwaite-Penniston on $\omega$}

Following Zwegers \cite{zwegers} we define
\begin{align*}
F_1(\tau)&:=q^{-\frac{1}{24}}\mockf(q),\\
F_2(\tau)&:=2q^\frac{1}{3}\omega(q^\frac{1}{2}),\\
F_3(\tau)&:=2q^\frac{1}{3}\omega(-q^\frac{1}{2}),
\end{align*}
and the vector-valued function
$\mathbf{F}(\tau):=(F_1(\tau),F_2(\tau),F_3(\tau))^T$ for $\tau \in
\H$. For $z \in \C$ define
\begin{align*}
G_1(z)&:=-\sum_{n \in \Z} (n+\tfrac{1}{6})e^{3\pi i\left(n+\tfrac{1}{6}\right)^2z},\\
G_2(z)&:=\sum_{n \in \Z} (-1)^n (n+\tfrac{1}{3})e^{3\pi i\left(n+\tfrac{1}{3}\right)^2z},\\
G_3(z)&:=\sum_{n \in \Z} (n+\tfrac{1}{3})e^{3\pi
i\left(n+\tfrac{1}{3}\right)^2z},
\end{align*}
and finally for $\tau \in \H$:
\[ \mathbf{G}(\tau):=2 i \sqrt{3} \int_{-\overline{\tau}}^{i \infty} \frac{(G_1(z),G_2(z),G_3(z))^T}{\sqrt{-i(\tau + z)}} dz.\]

We can now state the main theorem of \cite{zwegers}.
\begin{theorem}[\cite{zwegers} Theorem 3.6]\label{thm_zwegersomega}
The function $\mathbf{H}(\tau):=\mathbf{F}(\tau)-\mathbf{G}(\tau)$
is a vector valued real analytic modular form of weight
$\frac{1}{2}$ satisfying
\[ \mathbf{H}(\tau+1)=\left(
     \begin{array}{ccc}
       \zeta_{24}^{-1} & 0 & 0 \\
       0 & 0 & \zeta_3 \\
       0 & \zeta_3 & 0 \\
     \end{array}
   \right)\mathbf{H}(\tau),
\]
and
\[ \frac{1}{\sqrt{-i\tau}}\mathbf{H}(-\tfrac{1}{\tau})=\left(
                                              \begin{array}{ccc}
                                                0 & 1 & 0 \\
                                                1 & 0 & 0 \\
                                                0 & 0 & -1 \\
                                              \end{array}
                                            \right)\mathbf{H}(\tau).
\]
Furthermore $\mathbf{H}$ is annihilated by $\Delta_\frac{1}{2}$.
\end{theorem}
We could now proceed analogously as in the the study of $\our$.
However, it turns out that the fact that $H_2(\tau)$ is not mapped
to a multiple of itself under translation, causes some problems.
This problem is circumvented if we study instead the function
$H_2(6\tau)$. For this function the transformation properties have
already been studied completely.

\begin{theorem}[\cite{garthpennis}, Corollary 4.2]\label{thm_garthwaite}
The function $H_2(6\tau)$  is a harmonic weak Maass form of weight
$\frac{1}{2}$ on $\Gamma_0(144)$ with respect to the
$\theta$-multiplier with character $\chi_{12}$ (recall this
definition from Section 1).
\end{theorem}

Next we find the Fourier expansion of $H_2(6\tau)$.

\begin{lem}\label{lem_omegafourier}
The Fourier expansion of $H_2(6\tau)$ has the form
\[H_2(6\tau)=2q^2\omega(q^3)+\frac{1}{\sqrt{\pi}}\sum_{n\equiv 1 \imod 3}^\infty 2 (-1)^\frac{n-1}{3} q^{-n^2} \Gamma\left(\tfrac{1}{2}, 4 \pi n^2 y\right).\]
\end{lem}

Still we need to know how $H_2(\tau)$ behaves under all
transformations of $\SL_2(\Z)$.

\begin{prop}\label{prop_omegamap}
Let $M=\left(%
\begin{smallmatrix}
  a & b \\
  c & d \\
\end{smallmatrix}%
\right) \in \SL_2(\Z)$. Then
$\frac{1}{\sqrt{c\tau+d}}H_2\left(\frac{a\tau+b}{c\tau +d}\right)$
is a constant multiple of
\begin{enumerate}
\item $H_1(\tau)$, if $b$ is odd and $a$ is even,
\item $H_2(\tau)$, if $b$ is even and $a$ is odd,
\item $H_3(\tau)$, if $b$ and $a$ are odd.
\end{enumerate}
\end{prop}
\subsection{Congruences for $\omega$}
The coefficients of $\omega$ for which we expect the congruences as
stated in Theorem \ref{thm_mainthm} are exactly the Fourier
coefficients of $H_2(6\tau)$ at those powers of $q$ which have the
form $q^{120n+83}$ and $q^{120n+107}$. In order to sieve out these
coefficients, we do not apply the sieving operator $U_{r,m}$
directly but we sieve with twists of quadratic characters, because
this yields modular forms on bigger groups, which is favorable for
for computational reasons. Consider the characters
$\chi_3(n):=\left( \frac{n}{3}\right)$ and $\chi_5(n):=\left(
\frac{n}{5}\right)$ and the quadratic characters $\imod{8}$ which
are given as follows
\begin{center}
\begin{tabular}{c|cccccccc}
& 0 & 1 & 2 & 3 & 4 & 5 & 6 & 7 \\
\hline
$\chi_8^{(0)}$ & 0 & 1 & 0 & 1 & 0 & 1 & 0 & 1 \\
$\chi_8^{(1)}$ & 0 & 1 & 0 & 1 & 0 & -1 & 0 & -1 \\
$\chi_8^{(2)}$ & 0 & 1 & 0 & -1 & 0 & 1 & 0 & -1 \\
$\chi_8^{(3)}$ & 0 & 1 & 0 & -1 & 0 & -1 & 0 & 1 \\
\end{tabular}
\end{center}

Then it is easy to verify that

\[ \frac{1}{2}\left(\chi_3(n)-1\right)\chi_3(n)=\left\{
                                                                                          \begin{array}{ll}
                                                                                            1 & \hbox{if }n \equiv 2 \imod{3}, \\
                                                                                            0 & \hbox{if } n \equiv 0,1
                                                                                            \imod{3}.
                                                                                          \end{array}
                                                                                        \right.\]

\[ \frac{1}{2}\left(1+\chi_5(n)\right)\chi_5(n)=\left\{
                                                                                          \begin{array}{ll}
                                                                                            1 & \hbox{if }n \equiv 2,3 \imod{5}, \\
                                                                                            0 & \hbox{if } n \equiv 0,1,4
                                                                                            \imod{5}.
                                                                                          \end{array}
                                                                                        \right.
\]

\[ \frac{1}{4}\left(\chi_8^{(0)}(n)+\chi_8^{(1)}(n)-\chi_8^{(2)}(n)-\chi_8^{(3)}(n)\right)=\left\{
                                                                                          \begin{array}{ll}
                                                                                            1 & \hbox{if }n \equiv 3 \imod{8}, \\
                                                                                            0 & \hbox{if } n \not\equiv 3
                                                                                            \imod{8}.                                                                                                                                                                                      \end{array}
                                                                                        \right.\]

The product of the above functions is exactly the characteristic
function of the set $\{120n+83,120n+107| n \in \Z \}$. It is easy to
prove that the twisted forms of $H_2(6\tau)$ are harmonic weak Maass
form for a certain subgroup of $\SL_2(\Z)$ which may be explicitly
computed. More precisely we obtain the following result.

\begin{prop}\label{prop_omegasieve}
The function \[ \sum_{n \equiv 27,35 \imod{40}}
a_\omega(n)q^{3n+2}\] is a weakly holomorphic modular form of weight
$\frac{1}{2}$ for $\Gamma_0(86400)$ and $\theta$-multiplier with
character $\chi_{12}$.
\end{prop}
In order to determine the behavior of this function at the cusps of
$\Gamma_0[86400]$ we prove the following analogue of Lemma
\ref{lem_cesaromatrixlemma}.
\begin{lem}\label{lem_omegamatrixlemma}
Let $\left(\begin{smallmatrix}  a & b \\  c & d
\\ \end{smallmatrix}\right) \in  \SL_2(\Z)$ and suppose $0 \leq s < m$ and $6|m$. Set  $l:=\gcd(\frac{cm}{6},sc+am)$. Then define
$\tilde{a}:=\frac{sc+am}{l}$ and $\tilde{c}:=\frac{cm}{6l}$. Let
$\tilde{d}$ be such that $\tilde{a}\tilde{d}-1 \equiv 0
\imod{\tilde{c}}$ and
$\tilde{b}:=\frac{\tilde{a}\tilde{d}-1}{\tilde{c}}$. Finally define
$t:=\frac{dl-\tilde{d}m}{c}$. Then for all $\tau \in \H$ and
$\left(\begin{smallmatrix} \tilde{a} & \tilde{b} \\  \tilde{c} &
\tilde{d}
\\ \end{smallmatrix}\right) \in  \SL_2(\Z)$ we have
\[  \left(
            \begin{array}{cc}
              6 & 0 \\
              0 & 1 \\
            \end{array}
          \right) \left(%
\begin{array}{cc}
  1 & \frac{s}{m} \\
  0 & 1 \\
\end{array}%
\right)\left(%
\begin{array}{cc}
  a & b \\
  c & d \\
\end{array}%
\right)\tau\ = \left(
                         \begin{array}{cc}
                            \tilde{a} & \tilde{b} \\
                            \tilde{c} & \tilde{d} \\
                          \end{array}
                        \right)\left(
     \begin{array}{cc}
       l & t \\
       0 & \frac{m^2}{6l} \\
     \end{array}
   \right)\tau. \]
\end{lem}
Using this lemma and the transformation properties of $H_2(\tau)$ we
can deduce the following lemma with a proof completely analogous to
the proof of Lemma \ref{lem_cesarocuspbehaviour}.
\begin{lem}\label{lem_omegacuspbehaviour}
Let $\frac{a}{c}$ be a cusp and suppose that
$\left(\begin{smallmatrix} a & b \\  c & d
\\ \end{smallmatrix}\right) \in  \SL_2(\Z)$ is a matrix which maps
$\infty$ to this cusp. Let $m$ be some integer divisible by $6$ and
$0 \leq s < m$.

\begin{enumerate}
\item The Fourier expansion of the holomorphic part of $H_2(6(\tau+\frac{s}{m}))$
 at this cusp is the Fourier expansion at $\infty$ of
$H_2(\tau)$ under the following transformation
\[ \left(
                         \begin{array}{cc}
                            \tilde{a} & \tilde{b} \\
                            \tilde{c} & \tilde{d} \\
                          \end{array}
                        \right)\left(
     \begin{array}{cc}
       l & t \\
       0 & \frac{m^2}{6l} \\
     \end{array}
   \right)\]
with notations as in Lemma \ref{lem_omegamatrixlemma}.
\item The Fourier expansion of the holomorphic part of $H_2(6\tau)$ at this cusp
starts up (to a constant with)
\begin{enumerate}
\item $q^{-\frac{l^2}{4m^2}}$, if $\tilde{a}$ is even and $\tilde{b}$ is
odd.
\item $q^{\frac{l^2}{18m^2}}$, if $\tilde{a}$ is odd and $\tilde{b}$ is even.
\item $q^{\frac{l^2}{18m^2}}$, if $\tilde{a}$ is odd and $\tilde{b}$ is odd.
\end{enumerate}
\end{enumerate}
\end{lem}
Our next task is to construct a suitable cusp form which we multiply
with
\[ \sum_{n \equiv 27,35 \imod{40}} a_\omega(n)q^{3n+2} \] in order
to get a modular form that is holomorphic at the cusps to which we
may apply Sturm's Theorem. Our result is as follows.

\begin{prop}\label{prop_omegacuspform}
The function  $\eta^{240}(120\tau)\Delta^2(\tau)$ is a cusp form of
weight $144$ for $\Gamma_0(86400)$. Furthermore the function
\[ \eta^{240}(120\tau)\Delta^2(\tau) \sum_{n \equiv 27,35 \imod{40}}
a_\omega(n)q^{3n+2}\] is holomorphic at every cusp of
$\Gamma_0(86400)$.
\end{prop}
The proof is analogous to the proof of Proposition
\ref{prop_cesarocuspform}. As an easy consequence of this
proposition and Sturm's theorem we get the following result.

\begin{prop}\label{prop_omegasturm}
The congruences
\[ a_\omega(40n+27) \equiv a_\omega(40n+35) \equiv 0 \pmod{5}\]
hold for all $n$ if they hold for all $n$ for which $40n+27$ or $
40n+35$ is less than or equal to 832$\;$320.
\end{prop}
Using a computer program, which can be found on the author's
homepage\\
http://www.mi.uni-koeln.de/\~{ }mwaldher/, we computed enough
coefficients of both $\omega$ and $\our$ in order to deduce Theorem
\ref{thm_mainthm} from Proposition \ref{prop_cesarosturm} and
Proposition \ref{prop_omegasturm}.

\end{document}